\documentclass[a4paper,leqno,11pt,reqno]{amsart}
\usepackage{epsf,graphicx,epsfig}
\usepackage{amscd}
\usepackage{chngcntr}
\usepackage{amsmath,latexsym,amssymb,amsthm}
\usepackage[nospace,noadjust]{cite}
\usepackage{textcomp}
\usepackage{setspace,cite}
\usepackage{lscape,fancyhdr,fancybox}
\usepackage[all,cmtip]{xy}
\counterwithout{equation}{section}

\usepackage{graphicx}

\usepackage{color}
\usepackage{url}
\usepackage{enumitem}
\usepackage[mathscr]{euscript}

  \theoremstyle{plain}
\newtheorem{theorem}{Theorem}[section]
\newtheorem{corollary}[theorem]{Corollary}

\newtheorem{proposition}[theorem]{Proposition}
\newtheorem{example}[theorem]{Example}
\theoremstyle{definition}
\newtheorem{definition}[theorem]{Definition}
\theoremstyle{remark}
\newtheorem{remark}[theorem]{Remark}
\numberwithin{equation}{section}

\newcommand{\D}{\mathcal{L}(A)}

\begin{document}

\title{Cohomology and deformations of Filippov algebroids}

\author{Satyendra Kumar Mishra}
\address{Statistics and Mathematics Unit, Theoretical Statistics and Mathematics Division, Indian Statistical Institute, 8th Mile, Mysore Road, RVCE Post, Bangalore -560059, INDIA.}
\email{satyamsr10@gmail.com}

\author{Goutam Mukherjee}
\address{Statistics and Mathematics Unit, Theoretical Statistics and Mathematics Division, Indian Statistical Institute, 203 Barrackpore Trunk Road
Kolkata 700108, INDIA.}
\email{goutam@isical.ac.in}

\author{Anita Naolekar}
\address{Statistics and Mathematics Unit, Theoretical Statistics and Mathematics Division, Indian Statistical Institute, 8th Mile, Mysore Road, RVCE Post, Bangalore -560059, INDIA.}
\email{anita@isibang.ac.in}
\footnote{Corresponding author: Satyendra Kumar Mishra (email: satyamsr10@gmail.com).\\
AMS Mathematics Subject Classification (2010): $17$A$30,$ $17$B$55,$ $17$B$99$.}

 \subjclass[2000]{}
\keywords{Filippov algebroid; Filippov algebra; deformation; differential graded Lie algebra; deformation cohomology.}
\date{}
%
\begin{abstract}
In this article, we study the deformations of Filippov algebroids. First, we define a differential graded Lie algebra for a Filippov algebroid by introducing the notion of Filippov multiderivations for a vector bundle. We then discuss deformations of a Filippov algebroid in terms of low-dimensional cohomology associated with this differential graded Lie algebra. We define Nijenhuis operators on Filippov algebroids and characterize trivial deformations of Filippov algebroids in terms of these operators. Finally, we define finite order deformations and discuss the problem of extending a given finite order deformation to a deformation of a higher order.
\end{abstract}
\maketitle

\section{Introduction}
 The notion of $n$-ary operations first appeared in the context of cubic matrices introduced by Cayley in the nineteenth century. M. Kapranov, M. Gelfand, A. Zelevinskii, and N.P. Sokolov also considered the notion of cubic matrices in their works \cite{Kapranov, Sokolov}. In the literature, $n$-ary algebraic structures appeared naturally in theoretical and mathematical physics. The discovery of the Nambu mechanics \cite{Nambu}, and the work of S. Okubo on
Yang-Baxter equations \cite{Okubo} propagated a significant development on $n$-ary algebras. 

There are two $n$-ary generalizations of Lie algebras arising from different interpretations of the standard Jacobi 
identity in the $n$-ary case. Both of these generalizations have quite different applications and properties. In \cite{Filippov}, V.T. Filippov first considered $n$-ary operations $[~~,\ldots,~]$ on a vector space $V$, which satisfy the following generalized Jacobi identity
\begin{equation}\label{Jacobi1}
[a_1, \ldots, a_{n-1}, [b_1, \ldots, b_n]]= \sum_{i=1}^n [b_1, \ldots, b_{i-1}, [a_1, \ldots, a_{n-1}, b_i],b_{i+1}, \ldots, b_n]
\end{equation}
for all $a_i, b_j \in V$.
The above identity implies that $[a_1,\ldots, a_{n-1},~]$ is a derivation of the $n$-ary bracket for any $a_1,\ldots, a_{n-1}\in V$. These generalized Lie algebras are
known as Filippov algebras and the above generalized Jacobi identity is called the Fundamental identity or the Filippov identity. 

There is another natural generalization of the Jacobi identity considered by P. Michor and A. Vinogradov. This version provides an alternative approach to $n$-ary Lie algebras (for precise definition see \cite{MiVi}). In \cite{VinVin}, A. Vinogradov and M. Vinogradov introduced a three-parameter family of multiple Lie algebras to unify the above mentioned $n$-ary Lie algebra structures. 

In 1973, Y. Nambu proposed a generalized Hamiltonian  
dynamics in \cite{Nambu}, also known as ``Nambu Mechanics" by replacing the standard Poisson bracket by a ternary 
one as follows 
$$\{f_1,f_2,f_3\}=det\Bigg(\frac{\partial f_i}{\partial x_j}\Bigg)$$
for all $f_1,f_2,f_3\in C^{\infty}(\mathbb{R}^3)$. This ternary bracket satisfies the identity \eqref{Jacobi1}. Surprisingly for a long time, it remained unnoticed that the corresponding $n$-ary generalization of the standard Poisson bracket satisfies the generalized Jacobi identity \eqref{Jacobi1}. In \cite{SaVa}, D. Sahoo and M. C. Valsakumar noticed this fact in 1992. V.T. Filippov also observed it separately as an example of Filippov algebras in \cite{Filippov}. This observation was the starting point of the work of L. Takhtajan in \cite{Takht}. He systematically developed the fundamentals of Nambu mechanics by introducing $n$-ary Poisson manifolds or ``Nambu-Poisson manifolds" ($n$-ary bracket satisfying the identity \eqref{Jacobi1}). In the subsequent years, several articles (e.g. \cite{R1, R2, R3, R4, R5, R6, R7}) extensively studied Nambu mechanics and Nambu-Poisson manifolds.  In particular, J. Grabowski and G. Marmo discussed the relationship between linear Nambu-Poisson structures and  Filippov algebras. In this process, they introduced an $n$-ary generalization of Lie algebroids in \cite{GM} (known as Filippov algebroid of order $n$).

This paper aims to study deformations and their controlling cohomology of Filippov algebroids. The deformation of an algebraic or analytical object is a fundamental tool to study the object. The deformation theory of algebraic structures was first introduced by M. Gerstenhaber in his outstanding work for associative algebras \cite{Ger63}-\cite{Ger74}. A. Nijenhuis and R. W. Richardson, Jr. extended this theory to Lie algebras in \cite{NR66, NR67a, NR67b}. The deformation of algebraic structures is controlled by certain cohomology known as ``deformation cohomology". For  instance, in the case of associative algebras, the deformation cohomology is given by the Hochschild cohomology, and in the Lie algebra case, it is the Chevalley-Eilenberg cohomology (with coefficients in the adjoint representations). Unlike Lie algebras, geometric structures such as Lie algebroids \cite{CM08} do not admit adjoint representations (see \cite{TrMs} for more details). Since Lie algebroids are Filippov algebroids of order $2$, we cannot expect deformation cohomology with coefficients in an adjoint representation to control deformations of Filippov algebroids.


In \cite{CM08}, M. Crainic and I. Moerdijk introduced the notion of multiderivations of a vector bundle to obtain a differential graded Lie algebra (DGLA) associated to a Lie algebroid. They proved that this DGLA controls deformations (in the sense of \cite{CM08}) of a Lie algebroid.
In this article, we obtain suitable deformation cohomology for the more general case of Filippov algebroid. We introduce the notion of Filippov multiderivations for vector bundles and define a graded Lie algebra structure on the space of Filippov multiderivations. To this end, we describe Filippov algebroid structures on a vector bundle in terms of the graded Lie algebra of Filippov multiderivations on the vector bundle. Consequently, we associate a DGLA to a Filippov algebroid. In the particular case of Filippov algebras, this DGLA is the one that is described in \cite{Rotk05}. Moreover, in the particular case of Lie algebroids, this DGLA coincides with the one defined in \cite{CM08}. We introduce the notion of finite order one-parameter deformations of Filippov algebroids and interpret them in terms of the cohomology associated with the DGLA mentioned above. 

In \cite{Nijenhuis1}, I. Dorfman introduced and studied Nijenhuis operators on Lie algebras, which play an essential role in the study of integrability of nonlinear evolution equations. Nijenhuis operators for the case of Lie algebroids and $n$-Lie algebras are studied in \cite{LeftSym, LSBai}. In \cite{LSBai}, the authors studied trivial deformations of Filippov algebras in terms of Nijenhuis operators. Here, we introduce Nijenhuis operators on Filippov algebroids and characterize trivial deformations in terms of Nijenhuis operators. The following is a section-wise description of this paper.

In Section $2$, we recall definitions and examples of Filippov algebra and Filippov algebroid. We consider the base field to be $\mathbb{R}$, the field of real numbers. We also define the notion of $\mathcal{O}$-operator for a Filippov algebroid with representation on a vector bundle. In Section $3$, we define the notion of Filippov multiderivations of a vector bundle. Later on, we associate a differential graded Lie algebra (DGLA) to a Filippov algebroid. We denote by $H^*_{F}$, the associated cohomology to this DGLA of a  Filippov algebroid. In Section $4$, we define  one-parameter deformations of a Filippov algebroid. We characterize trivial deformations of a Filippov algebroid in terms of Nijenhuis operators. In the last section, we study the finite order deformations. We show that there is a bijective correspondence between $H^2_{F}$ and the equivalence classes of infinitesimal deformations of a Filippov algebroid. Finally, we discuss the problem of extending a finite order deformation to a deformation of the next higher order.

\section{Filippov algebroids}
In this section, we recall definitions and examples of Filippov algebras and Filippov algebroids from \cite{Filippov, GM}.

\begin{definition}
A Filippov algebra of order $n$ (or an $n$-Lie algebra) $L=(L, [~, \ldots, ~])$ over $\mathbb{R}$,  is a vector space $L$ together with a $\mathbb{R}$-multilinear map 
$$[~, \ldots,~]: L \times \stackrel{n}{\cdots} \times L\longrightarrow L$$
which is skew symmetric and satisfies the fundamental identity 
$$[a_1, \ldots, a_{n-1}, [b_1, \ldots, b_n]]= \sum_{i=1}^n [b_1, \ldots, b_{i-1}, [a_1, \ldots, a_{n-1}, b_i],b_{i+1}, \ldots, b_n],$$
for all $a_i, b_j \in L$.

\end{definition}

In \cite{GM}, J.\ Grabowski and G.\ Marmo introduced the notion of a Filippov algebroid as a generalization of Lie algebroids.

\begin{definition}
A Filippov algebroid $A=(A, [~, \ldots, ~], a)$ of order $n$ (or an $n$-Lie algebroid) over a smooth manifold $M$ is a smooth vector bundle $A$ over $M$ together with a Filippov algebra structure of order $n$ on the space $\Gamma A$ of smooth sections of $A$ given by the bracket 
$$[~, \ldots, ~]: \Gamma A \times \stackrel{n}{\cdots} \times \Gamma A \rightarrow \Gamma A,$$ 
 and a bundle map  $a: \wedge ^{n-1}A \rightarrow TM$ (the anchor map) satisfying the following conditions:\\
(a) For all $x_1, \ldots, x_{n-1}, y_1, \ldots, y_{n-1} \in \Gamma A,$ 
$$[a(x_1 \wedge \cdots \wedge x_{n-1}), a(y_1 \wedge \cdots \wedge y_{n-1})]= \sum_{i=1}^{n-1} a(y_1\wedge \cdots \wedge[x_1, \ldots, x_{n-1}, y_i]\wedge y_{i+1} \wedge \cdots\wedge y_{n-1}),$$
where the bracket on the left hand side is the usual Lie bracket on vector fields.\\
(b) For all $x_1, \ldots, x_{n-1}, y\in \Gamma A$ and $f \in C^\infty(M)$,\\
$$[x_1, \ldots, x_{n-1}, fy]= f[x_1, \ldots, x_{n-1}, y] + a(x_1\wedge \cdots \wedge x_{n-1})(f)y.$$
\end{definition}

\begin{example}
Any Lie algebroid is a Filippov algebroid of order $2$.
\end{example}

\begin{example}
Any Filippov algebra of order $n$ can be considered as a Filippov algebroid of order $n$ over a point. 
\end{example}

\begin{example}
Let $(L,[~,\ldots,~])$ be a Filippov algebra of order $n$, where the underlying vector space $L$ is of dimension $m$ with a basis $\{x_1,x_2,\ldots,x_m\}$. Let $c^k_{i_1,\ldots,i_n}$'s be structure constants of the Filippov algebra $(L,[~,\ldots,~])$, i.e.,
$$[x_{i_1},x_{i_2},\ldots,x_{i_n}]=\sum^m_{k=1} c^k_{i_1,\ldots,i_n} x_k.$$

Let $f\in C^{\infty}(\mathbb{R}^m)$. Then we can define a Filippov algebroid structure of order $n$ on the tangent bundle $T\mathbb{R}^m$. Here, the $n$-ary bracket is defined by 
$$[\frac{\partial}{\partial x_{i_1}},\frac{\partial}{\partial x_{i_2}},\ldots,\frac{\partial}{\partial x_{i_n}}]=f\sum^k_{i_1,\ldots,i_n}c^k_{i_1,\ldots,i_n}\frac{\partial}{\partial x_{k}}$$
and the anchor map is trivial. 
\end{example}

\begin{example}
The tangent bundle $T\mathbb{R}^m$ has a Filippov algebroid structure of order $n+1$ for $n\leq m$, where the bracket is given by 
$$[\frac{\partial}{\partial x_{i_1}},\frac{\partial}{\partial x_{i_2}},\ldots,\frac{\partial}{\partial x_{i_n}}]=0,$$
and the anchor map is defined by the tensor field $dx_1\wedge \cdots
\wedge dx_n\otimes \frac{\partial}{\partial x_1}$ 
\end{example}

\begin{definition}
Let $(A, [~, \ldots, ~], a)$ be a Filippov algebroid of order $n$ (over $M$). Then a representation of $A$ on a vector bundle $E$ over $M$ is an $\mathbb{R}$-linear map
$$\rho: \wedge^{n-1}\Gamma A\times \Gamma E \rightarrow \Gamma E $$ 
such that the following conditions are satisfied.
\begin{align}\label{Repcon1}
\nonumber
&\rho(x_1,\ldots,x_{n-1},\rho(y_1,\ldots,y_{n-1},\xi))-\rho(y_1,\ldots,y_{n-1},\rho(x_1,\ldots,x_{n-1},\xi))\\
=&\sum_{i=1}^{n-1}\rho(y_1,\ldots,y_{i-1},[x_1,\ldots,x_{n-1},y_i],y_{i+1},\ldots,y_{n-1},\xi);
\end{align}
\begin{align}\label{Repcon2}
\rho(x_1,\ldots,x_{n-2},[y_1,\ldots,y_{n-1},y_n],\xi))=\sum_{i=1}^n\rho(y_1,\ldots,\hat{y_i},\ldots,y_{n},\rho(x_1,\ldots,x_{n-2},y_i,\xi));
\end{align}
For any $1\leq i\leq n-1$ and $f\in C^{\infty}(M)$, 
\begin{align}
\rho(x_1,\ldots,x_{i-1},fx_i,x_{i+1},\ldots,x_{n-1},\xi)=f\rho(x_1,\ldots,x_i,\ldots,x_{n-1},\xi);
\end{align}
\begin{align}
\rho(x_1,\ldots,x_{n-1},f\xi)=f\rho(x_1,\ldots,x_{n-1},\xi)+a(x_1,\ldots,x_{n-1})(f).\xi.
\end{align}
Here, $x_i,y_i\in \Gamma A$ for $1\leq i\leq n-1$ and $\xi\in \Gamma E$. Also, note that conditions \eqref{Repcon1} and \eqref{Repcon2} imply that $\rho$ is a representation of Filippov algebra $(\Gamma A,[~,\ldots,~])$ on the vector space $\Gamma E$ (for details see \cite{LSBai}).

\end{definition}

Let $(A, [~, \ldots, ~], a)$ be a Filippov algebroid (order $n$) with a representation $\rho$ on a vector bundle $E$. Then we get a semidirect product Filippov algebroid $(A\rtimes_{\rho} E,[~,\ldots,~]_{\rho},a_{\rho})$ of order $n$, where 
\begin{enumerate}
\item the bracket $[~,\ldots,~]$ is given by 
$$[x_1+\xi_1,\ldots,x_{n}+\xi_{n}]=[x_1,\ldots,x_n]+\sum_{i=1}^n(-1)^{n-i}\rho(x_1,\ldots,\hat{x_i},\ldots,x_n,\xi_i),$$
\item and the anchor is given by 
 $$a_{\rho}(x_1+\xi_1,\ldots,x_{n-1}+\xi_{n-1})=a(x_1,\ldots,x_{n-1}).$$
\end{enumerate}

\begin{definition}\label{O-operator}
Let $(A, [~, \ldots, ~], a)$ be a Filippov algebroid of order $n$ (over $M$) with a representation $\rho$ on a vector bundle $E$ over $M$, then a bundle map $T:E\rightarrow A$ is called an $\mathcal{O}$-operator if the following identity holds.
$$[T\xi_1,\ldots, T\xi_n]=\sum_{i=1}^n (-1)^{n-i} T\big(\rho(T\xi_1,\ldots,\hat{T\xi_i},\ldots,T\xi_{n},\xi_i) \big)$$
for all $\xi_1,\ldots,\xi_n\in \Gamma E$.

\end{definition}

\begin{remark}
For $n=2$, the Definition \ref{O-operator} coincides with the notion of $\mathcal{O}$-operators on a Lie algebroid (see \cite{LeftSym}). 

In the particular case when $A$ is a vector bundle over a point, the Filippov algebroid is simply an $n$-Lie algebra and Definition \ref{O-operator} coincides with Definition $5$ in \cite{LSBai}.
\end{remark}

\section{Multiderivations on Filippov algebroids of order $n$}
Let $A=(A, [~, \ldots, ~], a)$ be a Filippov algebroid. The following proposition can be derived from Proposition $2$ of \cite{M}.
\begin{proposition}
Let $\mathcal{L}(A) := \Gamma A \wedge \cdots \wedge \Gamma A$ $(n-1$ copies$)$. Then $\D$ is a Leibniz algebra, where the bracket on $\D$ is defined by 
$$\begin{array}{ll}
&[x_1\wedge \cdots \wedge x_{n-1}, y_1\wedge \cdots \wedge y_{n-1}]\\
= &\sum_{i=1}^{n-1} y_1\wedge \cdots \wedge y_{i-1}\wedge [x_1, \ldots, x_{n-1}, y_i]\wedge y_{i+1}\wedge\cdots \wedge y_{n-1},
\end{array}$$ for all $x_i, y_j \in \Gamma A$.
\end{proposition}

\begin{definition}
Let $A=(A, [~, \ldots, ~], a)$ be a Filippov algebroid. A derivation of the Filippov algebroid $A$ is a linear map $D: \Gamma A \rightarrow \Gamma A$, which is a derivation of the underlying vector bundle $A$ as well as a derivation of the $n$-Lie algebra bracket, i.e.
\begin{enumerate}
\item there exists $\sigma_D \in \chi(M)$ such that 
$$D(fX)= f D(X) + \sigma_D(f) X \quad \mbox{ for all } f\in C^\infty(M),~X \in \Gamma A;$$

\item for all $x_1,\ldots,x_n\in \Gamma A$, we have 
$$D[x_1,\ldots, x_n]=\sum_{i=1}^n[x_1,\ldots, D(x_i),\ldots, x_n].$$
\end{enumerate} 
 \end{definition}

We denote the set of all derivations of the vector bundle $A$ by $\mbox{Der}(A)$. At the same time, the set of all derivations of the Filippov algebroid $(A, [~, \ldots, ~], a)$ is denoted by $\mbox{Der}_{F}(A)$. 

\begin{definition}
Let $A$ be a vector bundle over smooth manifold $M$. A map 
$$D: \otimes^{p-1} \D\otimes \D \wedge \Gamma A \longrightarrow \Gamma A$$ is called a Filippov multiderivation of degree $p$ ($p\geq 1$) on $A$ if there exists a $C^{\infty}(M)$-multilinear map   
$$\sigma_D: \otimes^p \D \rightarrow \chi(M)$$ 
such that the following identity holds.
$$D(X_1, \ldots, X_p, fz)= fD(X_1, \ldots, X_p, z) + \sigma_D(X_1, \ldots, X_p) (f) z$$ 
for $X_1,\ldots,X_p \in \D,~f\in C^{\infty}(M),$ and  $z \in \Gamma A$. The map $\sigma_D$ is said to be the `symbol' of the Filippov multiderivation $D$.  
\end{definition}

We denote the set of all Filippov multiderivations of the vector bundle $A$ of degree $p$ by $\mbox{Der}^p(A)$. By convention, $\mbox{Der}^0(A):= \mbox{Der}(A),$ and $\mbox{Der}^{-1}(A):= \D.$

\subsection{Graded Lie algebra structure on the space of Multiderivations}
Now, we define a bracket on the graded vector space $\{\mbox{Der}^*(A)\}_{*\geq -1}$, which makes it into a graded Lie algebra. Let $Sh (k, q)$ denote the set of $(k,q)$-shuffles in $S_{k+q}$ (the symmetric group on the set $\{1,\ldots,k+q\}$). Then, we define a circle product $D_1 \circ D_2 \in \mbox{Der}^{p+q}(A)$ for $D_1 \in \mbox{Der}^p(A)$ and $D_2 \in \mbox{Der}^q(A)$ as follows:
for any permutation $\sigma \in Sh(k,q)$ and for all $0 \leq k \leq p-1$, we define \\
\begin{align}\label{CPk}
&D_1 \circ_k^\sigma D_2(X_1, \ldots, X_{p+q}, z)\\\nonumber
=&(-1)^{kq}\sum_{s=1}^{n-1} D_1(X_{\sigma(1)}, \ldots, X_{\sigma(k)},X_{k+q+1}^1\wedge \cdots\wedge X_{k+q+1}^{s-1}\wedge D_2(X_{\sigma(k+1)}, \ldots, X_{\sigma(k+q)}, X_{k+q+1}^s) \\\nonumber
&\wedge X_{k+q+1}^{s+1}\wedge \cdots\wedge X_{k+q+1}^{n-1}, X_{k+q+2},\ldots, X_{p+q}, z).\nonumber
\end{align}
Also, for $k=p$ and $\sigma \in Sh(p, q)$, we define\\
\begin{align}\label{CP0}
 D_1 \circ_p^\sigma D_2(X_1, \ldots, X_{p+q}, z)
=  (-1)^{pq} D_1\big(X_{\sigma(1)}, \ldots, X_{\sigma(p)}, D_2(X_{\sigma(p+1)}, \ldots, X_{\sigma(p+q)}, z)\big).
\end{align}

Then, the circle product $D_1 \circ D_2 \in \mbox{Der}^{p+q}(A)$ for $D_1 \in \mbox{Der}^p(A)$ and $D_2 \in \mbox{Der}^q(A)$ is given by the following expression. 
\begin{equation}\label{Circle Product}
D_1\circ D_2 = \sum_{k=0}^p \sum_{\sigma \in Sh(k,q)} \mbox{sgn}(\sigma) D_1 \circ_k^\sigma D_2.
\end{equation}
Let us denote 
\begin{itemize}
\item For all $X_1, \ldots, X_{p+q}\in \D$ and $z\in A$,
\begin{equation*}
D_1 \circ_p D_2( X_1, \ldots, X_{p+q},z):= \sum_{\sigma \in Sh(p,q)}sgn(\sigma) D_1 \circ_p^\sigma D_2(X_1, \ldots, X_{p+q}, z).
\end{equation*}

\item For all $0\leq k\leq p-1$,
\begin{equation*}
D_1 \circ_k D_2( X_1, \ldots, X_{p+q},z)=\sum_{\sigma \in Sh(k,q)}sgn(\sigma) (D_1 \circ_k^{\sigma} D_2)( X_1, \ldots, X_{p+q},z).
\end{equation*}
\end{itemize}
Finally, we define a bracket on the space of multiderivations as follows:
\begin{equation}\label{Bracket}
[D_1, D_2]:= (-1)^{pq} D_1 \circ D_2 - D_2 \circ D_1.
\end{equation}

\begin{proposition}
The bracket as defined by equation \eqref{Bracket} makes $\mbox{Der}^*(A)$ into a graded Lie algebra.
\end{proposition}

\begin{proof}
It follows from Lemma $3$, \cite{Rotk05} that the bracket defined by equation \eqref{Bracket} is a graded Lie-bracket of degree $0$. 

Next, we show that for $D_1\in \mbox{Der}^p(A)$ and $D_2\in \mbox{Der}^q(A)$, the bracket $[D_1,D_2]\in \mbox{Der}^{p+q}(A)$. Let $f\in C^{\infty}(M)$, $X_1,\ldots,X_{p+q}\in \D$, and $z\in \Gamma A$. Then by the definition of $D_1 \circ_p D_2$, we obtain the following expression

\begin{align}\label{P1}
&(D_1 \circ_p D_2)( X_1, \ldots, X_{p+q},fz)\\ \nonumber
&=(-1)^{pq} \sum_{\sigma \in Sh(p,q)}sgn(\sigma) D_1 \circ_p^\sigma D_2(X_1, \ldots, X_{p+q}, fz)\\\nonumber
&=(-1)^{pq}  \sum_{\sigma \in Sh(p,q)} sgn(\sigma) D_1\big(X_{\sigma(1)}, \ldots, X_{\sigma(p)}, D_2(X_{\sigma(p+1)}, \ldots, X_{\sigma(p+q)}, fz)\big)\\\nonumber
&= (-1)^{pq}\sum_{\sigma \in Sh(p,q)} sgn(\sigma)\Big( D_1\big(X_{\sigma(1)}, \ldots, X_{\sigma(p)}, fD_2(X_{\sigma(p+1)}, \ldots, X_{\sigma(p+q)}, z)\big)\\\nonumber
&\quad\quad\quad\quad\quad +D_1\big(X_{\sigma(1)}, \ldots, X_{\sigma(p)}, \sigma_{D_2}(X_{\sigma(p+1)}, \ldots, X_{\sigma(p+q)})(f)z\big)\Big)\\\nonumber
&=A_1+A_2+B_1+B_2,
\end{align}
 \noindent where 
\begin{align*}
A_1 &  = (-1)^{pq}\sum_{\sigma \in Sh(p,q)} sgn(\sigma)~
f D_1\big(X_{\sigma(1)}, \ldots, X_{\sigma(p)}, D_2\big(X_{\sigma(p+1)}, \ldots, X_{\sigma(p+q)}, z\big)\big)\\\nonumber
A_2&=(-1)^{pq}\sum_{\sigma \in Sh(p,q)} sgn(\sigma)
~\sigma_{D_1}(X_{\sigma(1)}, \ldots, X_{\sigma(p)})(f)D_2 \big(X_{\sigma(p+1)}, \ldots, X_{\sigma(p+q)}, z\big) \\\nonumber
B_1&=(-1)^{pq}\sum_{\sigma \in Sh(p,q)} sgn(\sigma)~
\sigma_{D_2}(X_{\sigma(p+1)}, \ldots, X_{\sigma(p+q)})(f)D_1\big(X_{\sigma(1)}, \ldots, X_{\sigma(p)}, z\big)\\\nonumber
B_2&=(-1)^{pq}\sum_{\sigma \in Sh(p,q)} sgn(\sigma)~ \sigma_{D_1}\big(X_{\sigma(1)}, \ldots, X_{\sigma(p)}\big)\sigma_{D_2}\big((X_{\sigma(p+1)}, \ldots, X_{\sigma(p+q)}\big)(f)z
\end{align*}
Similarly, 
\begin{equation}\label{P2}
(D_2 \circ_q D_1)( X_1, \ldots, X_{p+q},fz)=C_1+C_2+D_1+D_2,
\end{equation}
where the expressions $C_1,~C_2,~D_1$ and $D_2$ are given as follows:
\begin{align*}
C_1&= (-1)^{pq}\sum_{\sigma \in Sh(q,p)} sgn(\sigma)
~fD_2\big(X_{\sigma(1)}, \ldots, X_{\sigma(q)}, D_1(X_{\sigma(q+1)}, \ldots, X_{\sigma(p+q)}, z)\big)\\\nonumber
C_2&=(-1)^{pq}\sum_{\sigma \in Sh(q,p)} sgn(\sigma)~\sigma_{D_2}(X_{\sigma(1)}, \ldots, X_{\sigma(q)})(f).D_1 \big(X_{\sigma(q+1)}, \ldots, X_{\sigma(p+q)}, z\big)\\\nonumber
D_1&=(-1)^{pq}\sum_{\sigma \in Sh(q,p)} sgn(\sigma)~
\sigma_{D_1}(X_{\sigma(q+1)}, \ldots, X_{\sigma(p+q)})(f).D_2\big(X_{\sigma(1)}, \ldots, X_{\sigma(q)}, z\big)\\\nonumber
D_2&=(-1)^{pq}\sum_{\sigma \in Sh(q,p)} sgn(\sigma)~ \sigma_{D_2}\big(X_{\sigma(1)}, \ldots, X_{\sigma(q)}\big)\sigma_{D_1}\big(X_{\sigma(q+1)}, \ldots, X_{\sigma(p+q)}\big)(f)z.\\\nonumber
\end{align*}
Let us observe that the following expressions hold.
\begin{enumerate}
\item $(-1)^{pq}A_2-D_1=0$, and 
\item $(-1)^{pq}B_1-C_2=0.$
\end{enumerate} 
Next, using the properties of multiderivations we obtain the following expressions:
\begin{align}\label{P3}
&\sum_{k=0}^{p-1}(D_1 \circ_k D_2)( X_1, \ldots, X_{p+q},fz)\\\nonumber
&=\sum_{k=0}^{p-1} \sum_{\sigma \in Sh(k,q)}sgn(\sigma) (-1)^{kq}\sum_{s=1}^{n-1} fD_1\Big(X_{\sigma(1)}, \ldots, X_{\sigma(k)},X_{k+q+1}^1\wedge \cdots\wedge X_{k+q+1}^{s-1}\wedge\\\nonumber
&\quad D_2(X_{\sigma(k+1)}, \ldots, X_{\sigma(k+q)}, X_{k+q+1}^s) 
\wedge X_{k+q+1}^{s+1}\wedge \cdots\wedge X_{k+q+1}^{n-1}, X_{k+q+2},\ldots, X_{p+q}, z\Big)\\\nonumber
&+\sum_{k=0}^{p-1} \sum_{\sigma \in Sh(k,q)}sgn(\sigma) (-1)^{kq}\sum_{s=1}^{n-1} \sigma_{D_1}\Big(X_{\sigma(1)}, \ldots, X_{\sigma(k)},X_{k+q+1}^1\wedge \cdots\wedge X_{k+q+1}^{s-1}\wedge\\\nonumber
&\quad D_2(X_{\sigma(k+1)}, \ldots, X_{\sigma(k+q)}, X_{k+q+1}^s) 
\wedge X_{k+q+1}^{s+1}\wedge \cdots\wedge X_{k+q+1}^{n-1}, X_{k+q+2},\ldots, X_{p+q}\Big)(f)z,
\end{align}
\newpage
and
\begin{align}\label{P4}
&\sum_{k=0}^{q-1}(D_2 \circ_k D_1)( X_1, \ldots, X_{p+q},fz)\\ \nonumber
&=\sum_{k=0}^{q-1} \sum_{\sigma \in Sh(k,q)}sgn(\sigma) (-1)^{kp}\sum_{s=1}^{n-1} fD_2\Big(X_{\sigma(1)}, \ldots, X_{\sigma(k)},X_{k+q+1}^1\wedge \cdots\wedge X_{k+p+1}^{s-1}\wedge\\\nonumber
&\quad D_1(X_{\sigma(k+1)}, \ldots, X_{\sigma(k+p)}, X_{k+p+1}^s) 
\wedge X_{k+p+1}^{s+1}\wedge \cdots\wedge X_{k+p+1}^{n-1}, X_{k+p+2},\ldots, X_{p+q}, z\Big)\\\nonumber
&+\sum_{k=0}^{q-1} \sum_{\sigma \in Sh(k,q)}sgn(\sigma) (-1)^{kp}\sum_{s=1}^{n-1} \sigma_{D_2}\Big(X_{\sigma(1)}, \ldots, X_{\sigma(k)},X_{k+p+1}^1\wedge \cdots\wedge X_{k+p+1}^{s-1}\wedge\\\nonumber
&\quad D_1(X_{\sigma(k+1)}, \ldots, X_{\sigma(k+p)}, X_{k+p+1}^s) 
\wedge X_{k+p+1}^{s+1}\wedge \cdots\wedge X_{k+p+1}^{n-1}, X_{k+p+2},\ldots, X_{p+q}\Big)(f)z.
\end{align}
Consequently, by using equations \eqref{P1}-\eqref{P4}, we get the following identity for the graded Lie bracket $[~,~]$:
\begin{align*}
[D_1,D_2](X_1,\ldots,X_{p+q},fz)=f[D_1,D_2](X_1,\ldots,X_{p+q},z)
+\sigma_{[D_1,D_2]}(X_1,\ldots,X_{p+q})(f)z,
\end{align*}
where the map $\sigma_{[D_1, D_2]}:\otimes^{p+q}\D\rightarrow \chi(M)$
is given by the following equation
\begin{equation}\label{Symbol of Bracket}
\sigma_{[D_1, D_2]} = (-1)^{pq} \sigma_{D_1} \odot D_2 - \sigma_{D_2} \odot D_1 + \{\sigma_{D_1}, \sigma_{D_2}\}.
\end{equation}
Here, the notation $\{\sigma_{D_1}, \sigma_{D_2}\}$ is defined by the following expression 
\begin{align}
\nonumber
\{\sigma_{D_1}, \sigma_{D_2}\}(X_1, \ldots, X_{p+q})
 = \sum_{\sigma \in Sh(p,q)}& sgn(\sigma) \Big(\sigma_{D_1}\big(X_{\sigma(1)}, \ldots, X_{\sigma(p)}\big)\sigma_{D_2}\big(X_{\sigma(p+1)}, \ldots, X_{\sigma(p+q)}\big)\\\nonumber
&-\sigma_{D_2}\big(X_{\sigma(p+1)}, \ldots, X_{\sigma(p+q)}\big)\sigma_{D_1}\big(X_{\sigma(1)}, \ldots, X_{\sigma(p)}\big)\Big)
\end{align}
 and the notation $\sigma_{D_1}\odot D_2$ is defined by the following expression
\begin{align*}
&(\sigma_{D_1} \odot D_2)(X_1, \ldots, X_{p+q})\\ 
&=\sum_{k=0}^{p-1} \sum_{\sigma \in Sh(k,q)}sgn(\sigma) (-1)^{kq}\sum_{s=1}^{n-1} \sigma_{D_1}\Big(X_{\sigma(1)}, \ldots, X_{\sigma(k)},X_{k+q+1}^1\wedge \cdots\wedge X_{k+q+1}^{s-1}\wedge\\\nonumber
&\quad D_2(X_{\sigma(k+1)}, \ldots, X_{\sigma(k+q)}, X_{k+q+1}^s) 
\wedge X_{k+q+1}^{s+1}\wedge \cdots\wedge X_{k+q+1}^{n-1}, X_{(k+q+2)},\ldots, X_{(p+q)}\Big)
\end{align*}
for all $X_1,\ldots,X_{p+q}\in \D$. Therefore, for any $D_1\in \mbox{Der}^p(A)$ and $D_2\in \mbox{Der}^q(A)$, the bracket $[D_1,D_2]\in \mbox{Der}^{p+q}(A)$ with symbol $\sigma_{[D_1,D_2]}$. Hence, the space of multiderivation of the vector bundle $A$ is closed under the graded Lie bracket given by the equation \eqref{Bracket}.

\end{proof}

The following Lemma shows that Filippov algebroids are canonical structures for the graded Lie bracket given by equation \eqref{Bracket}. 

\begin{proposition}\label{canonical structure}
Let $A$ be a vector bundle over $M$. Then there is a one-to-one correspondence between elements $\varphi\in \mbox{Der}^1(A)$ satisfying $[\varphi, \varphi]=0$ and Filippov algebroid structures on $A$.
\end{proposition}

\begin{proof}
First let us assume that $(A, [~,\ldots,~], a)$ be a Filippov algebroid. Let us define a map $\varphi:\D \wedge \Gamma A\longrightarrow \Gamma A$ by 
$$\varphi(X, y) := [X^1,\ldots,X^{n-1}, y] \quad\mbox{for}~~X\in \D,~ y \in L.$$ 
Then by definition of a Filippov algebroid it follows that $\varphi:= [~,\ldots,~]$ is a multiderivation with the symbol $\sigma_\varphi=a:\D\rightarrow \Gamma TM$. Thus, $\varphi\in \mbox{Der}^1(A)$. Next, from the definition of the product $\circ$, we obtain the following expression.
\begin{align}\label{canonical}
\nonumber
&\varphi\circ \varphi(X_1, X_2, z)\\\nonumber
&=\sum_{\sigma \in Sh(0, 1)} \mbox{sgn}(\sigma) \varphi \circ_0^\sigma \varphi(X_1, X_2, z) + (-1)\sum_{\sigma \in Sh(1, 1)} \mbox{sgn}(\sigma) \varphi \circ_1^\sigma \varphi(X_1, X_2, z)\\\nonumber
&=\sum_{s=1}^{n-1} \varphi(X_2^1, \ldots, X_2^{s-1}, \varphi(X_1, X_2^s), X_2^{s+1}, \ldots, X_2^{n-1},z) - \varphi(X_1, \varphi(X_2, z)) + \varphi(X_2, \varphi(X_1, z)) \\
&= \sum_{s=1}^{n-1} [X_2^1, \ldots, X_2^{s-1}, [X_1, X_2^s], X_2^{s+1}, \ldots, X_2^{n-1},z]- [X_1, [X_2, z]]+ [X_2,[X_1,z]].
\end{align}
Since the bracket $[~,\ldots,~]$ satisfies the fundamental identity, we get $\varphi\circ \varphi=0$, i.e. $$[\varphi,\varphi]=-2\varphi\circ \varphi=0.$$

Conversely, let $A$ be a vector bundle over $M$ and $\varphi\in Der^1(A)$ satisfying $[\varphi,\varphi]=0$. Then let us define a Filippov algebroid structure on $A$ as follows:
\begin{itemize}
\item the bracket is given by $[~,\ldots,~]:=D$. It follows by equation \eqref{canonical} that the bracket $[~,\ldots,~]$ satisfies the fundamental identity since $\varphi\circ \varphi=-\frac{1}{2}[\varphi,\varphi]=0$. 
\item the anchor map $a:=\sigma_\varphi$ is given by the symbol of the multiderivation $\varphi$.
\end{itemize}
\end{proof} 

Therefore, using the canonical structure of Filippov algebroids with respect to the graded Lie bracket given by the equation \eqref{Bracket}, we define a cochain complex $(C^*_F(A),\delta_F)$, where 
$$C^k_F(A)=Der^{k-1}(A)\quad \mbox{and }\quad \delta_F(\psi)=[\varphi,\psi] \mbox{ for any } \psi\in Der^{*}(A).$$
Here, $\varphi\in Der^1(A)$ satisfying $[\varphi,\varphi]=0$, is the corresponding $1$-multiderivation to the Filippov algebroid structure on the bundle $A$. We denote the associated cohomology of the cochain complex $(C^*_F(A),\delta_F)$ by $H^*_F(A)$.

\begin{example}
For $n=2$, the Filippov algebroid $A$ of order $n$ is simply a Lie algebroid. Moreover, the cochain complex $(C^*_F(A),\delta_F)$ and the cohomology $H^*_{F}(A)$ is the same as the deformation complex $ (C^*_{def}(A),\delta)$ and the cohomology $H^*_{def}(A)$ of the Lie algebroid $A$ (defined in \cite{CM08}).
\end{example}

\begin{example}

Let $(L,[~,\ldots,~])$ be an $n$-Lie algebra. Then $(L,[~,\ldots,~])$ is also a Filippov algebroid with trivial anchor. A Filippov multiderivation $\varphi$ of degree $p$ on $(L,[~,\ldots,~])$ is simply a linear map 
$$\varphi:\otimes^p(\wedge^{n-1}L)\wedge L\rightarrow L.$$
Therefore, 
$$Der^p(L)=C^p_{ad}$$
where, $C^p_{ad}$ is the space of $L$-valued $p$-cochains defined in Section $11.7$ of \cite{Survey}. Next, let us also recall from \cite{Survey} that the coboundary operator 
$$\delta:C^p_{ad}\rightarrow C^{p+1}_{ad}$$
for deformation cohomology of $n$-Lie algebra is given as follows.
 \begin{align*}
&\delta\varphi(X_1, \ldots, X_{p+1}, z)\\
=&\sum_{i=1}^{p+1}(-1)^{i} \psi(X_1, \ldots, \hat{X_i}, \ldots, X_p, [X_i, z])\\
&+ \sum_{1\leq i< j}^{p+1} (-1)^{i}\psi(X_1, \ldots, \hat{X_i}, \ldots, X_{j-1}, [X_i, X_j], X_{j+1}, \ldots, X_{p+1}, z)\quad \quad\quad \quad \\
&+ \sum_{i=1}^{p+1}(-1)^{i+1} [X_i, \psi(X_1, \ldots, \hat{X_i}, \ldots, X_{p+1}, z)]\\
&+(-1)^p \sum_{s=1}^{n-1}[X_{p+1}^1, \ldots, X_{p+1}^{s-1}, \psi(X_1, \ldots, X_p, X_{p+1}^s), X_{p+1}^{s+1}, \ldots, X_{p+1}^{n-1}, z]
\end{align*}
for all $X_1,\ldots,X_{p+1}\in \wedge^{n-1}L$ and $z\in L$. Here, 
$$[X_i, X_j]=\sum_{s=1}^{n-1}X_j^1\wedge X_j^2\wedge\cdots \wedge X_j^{s-1}\wedge[X_i,X_j^{s}]\wedge X_j^{s+1}\wedge\cdots\wedge X_j^{n-1}.$$
Now, let us consider the differential $\delta_F$ in the complex $(Der^*(L),\delta_F)$. Also, let $\varphi$ be the element in $\mbox{Der}^1(A)$ induced by the $n$-Lie bracket. Then for all $X_1,\ldots,X_{p+1}\in \wedge^{n-1}L$ and $z\in L$, from equation \eqref{Bracket}, we have the following expression.
\begin{align*}
&\delta_F \varphi(X_1, \ldots, X_{p+1}, z)\\
=&[\varphi, \varphi](X_1, \ldots, X_{p+1}, z)\\
= &((-1)^p \varphi\circ \psi-\psi \circ \varphi)(X_1, \ldots, X_{p+1}, z),
\end{align*}
i.e.,
\begin{align*}
&\delta_F \varphi(X_1, \ldots, X_{p+1}, z)\\
=&(-1)^p\Big\{\sum_{k=0}^1 (-1)^{kp}\sum_{\sigma \in Sh(k, p)} \mbox{sgn}(\sigma) \varphi \circ_k^\sigma \psi (X_1, \ldots, X_{p+1}, z)\Big\}\quad\quad\quad\quad\\
&-\Big\{ \sum_{k=0}^p (-1)^{k} \sum_{\sigma \in Sh(k,1)} \mbox{sgn}(\sigma) \psi \circ_k^\sigma \varphi(X_1, \ldots, X_{p+1}, z)\Big\}\quad\quad\quad\quad\\
=&\sum_{i=1}^{p+1}(-1)^{i-1} \psi(X_1, \ldots, \hat{X_i}, \ldots, X_{p+1}, [X_i, z])
\end{align*}
\begin{align*}
&+ \sum_{1\leq i< j}^{p+1} (-1)^{i+1}\psi(X_1, \ldots, \hat{X_i}, \ldots, X_{j-1}, [X_i, X_j], X_{j+1}, \ldots, X_{p+1}, z)\quad \quad\quad \quad \\
&+ \sum_{i=1}^{p+1}(-1)^{i+1} [X_i, \psi(X_1, \ldots, \hat{X_i}, \ldots, X_{p+1}, z)]\\
&(-1)^p \sum_{s=1}^{n-1}[X_{p+1}^1, \ldots, X_{p+1}^{s-1}, \psi(X_1, \ldots, X_p, X_{p+1}^s), X_{p+1}^{s+1}, \ldots, X_{p+1}^{n-1}, z]\\
=& \delta \psi (X_1, \ldots, X_{p+1}, z).
\end{align*}
Thus, the deformation complex $(C^*_{ad},\delta)$ of a Filippov algebra (see \cite{Survey}) coincides with the complex $(Der^*(L),\delta_F)$.

\end{example}

\section{Deformations of Filippov algebroids}
In this section, we discuss deformations of Filippov algebroids. We define Nijenhuis operators on Filippov algebroids and characterize trivial deformations in terms of these operators. 

Let $X=x_1\wedge\cdots\wedge x_{n-1}\in\D$; then the map $ad_X:\Gamma A\rightarrow \Gamma A$, defined by 
$$ad_X(y)=[x_1,\ldots, x_{n-1},y]\quad\mbox{for all } y\in \Gamma A,$$
is a derivation of the $n$-Lie bracket: 
$$[x_1, \ldots, x_{n-1},[y_1, \ldots, y_n]]= \sum_{i=1}^n [y_1, \ldots, y_{i-1}, [x_1, \ldots, x_{n-1}, y_i], y_{i+1},\ldots, y_n],$$
for all $y_i\in\Gamma A$. Moreover, since $A$ is a Filippov algebroid, the map $ad_X:\Gamma A\rightarrow \Gamma A$ for each $X\in \D$ is a derivation of Filippov algebroid. We call the set of derivations $\{ad_X: X\in \mathcal{L}(A)\}$ to be the space of inner derivation of the Filippov algebroid $A$.

Let us consider a $1$-cocycle $\psi\in Der(A)$, then $\delta_F(\psi)=[D,\psi]=0$, which implies that $\psi$ is a derivation of the Filippov algebroid $A$. Let $\psi$ be $1$-coboundary, i.e., $\psi=\delta_F(X)=ad_X$ for some $X\in \D$. Thus, $\psi$ is an inner derivation of the Filippov algebroid $A$. 
Therefore, 
$$H^1_F(A)=\frac{\mbox{Der}_F(A)}{\mbox{Space of inner derivations of } A}:=\mbox{ Space of Outer derivations of } A.$$ 

Let us recall the canonical structure (Proposition \ref{canonical structure}) of Filippov algebroids. We consider a Filippov algebroid structure $(A,[~,\ldots,~],a)$ on a vector bundle $A$ over $M$ as an element $\varphi\in Der^1(A)$ satisfying $[\varphi,\varphi]=0$.  
 
 \begin{definition}
Let a Filippov algebroid structure on a vector bundle $A$ over $M$ be given by an element $\varphi\in Der^1(A)$ satisfying $[\varphi,\varphi]=0$. Then a one-parameter deformation of this Filippov algebroid structure is a $t$-parametrized family of $\mathbb{R}$-bilinear map $\varphi_t:\D \wedge \Gamma (A)\rightarrow \Gamma (A),$ which is defined by $$\varphi_t(X,y)=\sum_{i=0}^{m} t^i \varphi_i(X,y) ~\mbox{for} ~~\varphi_0=\varphi,~ \varphi_i\in Der^1(A)~ \mbox{for}~ 1\leq i\leq m,~t\in \mathbb{R},$$ and satisfies 
$$[\varphi_t,\varphi_t]=0.$$
\end{definition}

Note that $\varphi_t(X,y)=\sum\limits _{i=0}^{m} t^i \varphi_i(X,y)$ is a 1-degree Filippov multiderivation of $A$ with the symbol $\sigma_{\varphi_t}:\D\rightarrow \chi(M)$, given by 
$$\sigma_{\varphi_t}(X)=\sum_{i=0}^m t^i \sigma_{\varphi_i}(X).$$
Moreover, since $\varphi_t$ satisfies $[\varphi_t,\varphi_t]=0$, it corresponds to a Filippov algebroid structure on the vector bundle $A$ over $M$. In particular, it yields a $t$-parametrized family of brackets $$[~,\ldots,~]_t:\Gamma A \times \stackrel{n}{\cdots} \times \Gamma A \rightarrow \Gamma A$$
 and a family of maps $$a_t: \D \rightarrow \chi(M)$$ which satisfy the following identities:
\begin{itemize}
\item $[x_1,\ldots,x_n]_t=[x_1,\ldots,x_n]+\sum\limits_{i=1}^{m}t^i\varphi_i(x_1,\ldots,x_n)$,
\item $a_t(x_1,\ldots,x_{n-1})=a(x_1,\ldots,x_{n-1})+\sum\limits_{i=1}^{m}t^i\sigma_{\varphi_i}(x_1,\ldots,x_{n-1})$
\end{itemize}
for all $x_1,\ldots,x_n\in \Gamma A$. The $t$-parametrized family $(A,[~,\ldots,~]_t,a_t)$ is called a $1$-parameter deformation of $(A,[~,\ldots,~],a)$ generated by $\varphi_1,\cdots, \varphi_{m}\in Der^1(A)$. 


\begin{theorem}\label{Infinitesimal}
The $t$-parametrized family $(A,[~,\ldots,~]_t,a_t)$ is a deformation of a Filippov algebroid $(A,[~,\ldots,~],a)$, generated by $\varphi_1,\ldots, \varphi_m\in Der^1(A)$ if and only if the following conditions hold.
\begin{enumerate}
\item $\delta_F(\varphi_1)=0$;
\item $\delta_F(\varphi_k)+\frac{1}{2}\sum\limits_{i=1}^{k-1}[\varphi_i,\varphi_{k-i}]=0$ for all $1< k\leq m$;
\item $\sum\limits_{i=k-m}^{m}[\varphi_i,\varphi_{k-i}]=0$ for all $m+1\leq k\leq 2m$.

\end{enumerate} 
\end{theorem}

\begin{proof}
Since $\varphi_1,\ldots,\varphi_m\in Der^1(A)$, we observe that the bracket $[~,\ldots,~]_t$ and the map $a_t$ satisfy the following properties.
\begin{enumerate}
\item $[x_1, \ldots, x_{n-1}, fy]_t= f[x_1, \ldots, x_{n-1}, y]_t + a_t(x_1\wedge \cdots \wedge x_{n-1})(f)y$ for all $f\in C^{\infty}(M)$ and $x_1,\ldots,x_{n-1}\in \Gamma A$.
\item The map $a_t: \wedge ^{n-1}A \rightarrow TM$ is a vector bundle morphism. 
\end{enumerate}
Also, by applying the fundamental identity and the property (1) we can deduce that for all $x_1, \ldots, x_{n-1}, y_1, \ldots, y_{n-1} \in \Gamma A,$ 
$$[a(x_1 \wedge \cdots \wedge x_{n-1}), a(y_1 \wedge \cdots \wedge y_{n-1})]= \sum_{i=1}^{n-1} a(y_1\wedge \cdots \wedge[x_1, \ldots, x_{n-1}, y_i]\wedge y_{i+1} \wedge \cdots\wedge y_{n-1}),$$
where the bracket on the left hand side is the usual Lie bracket on vector fields. Next, it follows by a straightforward calculation that the bracket $[~,\ldots,~]_t$ satisfies the fundamental identity if and only if the conditions (1)-(3) hold (see Proposition $1$, \cite{LSBai} for details).
\end{proof}

Let us denote a deformation $(A,[~,\ldots,~]_t,a_t)$ of the Filippov algebroid $(A,[~,\ldots,~],a)$ simply by the notation $A_t$. Let us consider deformations $A_t$ and $A_t^{\prime}:=(A,[~,\ldots,~]_t^{\prime},a_t^{\prime})$ of $(A,[~,\ldots,~],a)$, generated by $\varphi_1,\varphi_2,\cdots,\varphi_{m_1}\in Der^1(A)$, and $\varphi^{\prime}_1,\varphi^{\prime}_2,\cdots,\varphi^{\prime}_{m_2}\in Der^1(A)$, respectively. 

\begin{definition}
The deformations $A_t$ and $A_t^{\prime}$ are said to be equivalent if there exists a bundle map $N:A\rightarrow A$ such that we get a family of Filippov algebroid homomorphisms $\Phi_t:=Id+tN:A_t\rightarrow A_t^{\prime}$, which satisfy
\begin{equation}\label{Trivial1}
\Phi_t([x_1,\ldots,x_n]_t)=[\Phi_t(x_1),\Phi_t(x_2),\ldots,\Phi_t(x_n)]_t^{\prime}
\end{equation}
\begin{equation}\label{Trivial2}
a_t^{\prime}\big(\Phi_t(x_1),\Phi_t(x_2),\ldots,\Phi_t(x_n)\big)=a_t(x_1,\ldots,x_{n-1})
\end{equation}
\end{definition}

On comparing the coefficients of $t$ in the equation \eqref{Trivial1}, we get
 \begin{align}\label{Exp:1}
(\varphi_1-\varphi_1^{\prime})(x_1,\ldots,x_n)&=\sum^n_{i=1}[x_1,\ldots,N(x_i),\ldots,x_n]-N([x_1,\ldots,x_n]).
\end{align}
Thus, we obtain the following proposition.

\begin{proposition}\label{Eq1}
The infinitesimals of two equivalent deformations of $A_t$ belong to the same cohomology class.
\end{proposition}

\begin{proof}
Let the map $Id+tN:A\rightarrow A$ be an equivalence between the deformations $A_t$ and $A_t^{\prime}$, generated by $\varphi_1,\varphi_2,\cdots,\varphi_{m_1}\in Der^1(A)$ and $\varphi^{\prime}_1,\varphi^{\prime}_2,\cdots,\varphi^{\prime}_{m_2}\in Der^1(A)$ respectively. We observe that expression \eqref{Exp:1} implies 
$$\varphi_1-\varphi^{\prime}_1=\delta_F(N).$$
Thus, $\varphi_1$ and $\varphi^{\prime}_1$ represent the same cohomology class in $H^2_F (A)$.
\end{proof}

\begin{definition}
A deformation $A_t:=(A,[~,\ldots,~]_t,a_t)$ of Filippov algebroid $(A,[~,\ldots,~],a)$ is said to be trivial if there exist a bundle map $N:A\rightarrow A$ such that we get a family of Filippov algebroid homomorphisms 
$$\Phi_t=Id+tN: A_t\rightarrow (A,[~,\ldots,~],a).$$
\end{definition}

Let $A_t:=(A,[~,\ldots,~]_t,a_t)$ be a trivial deformation of Filippov algebroid $(A,[~,\ldots,~],a)$ generated by $\varphi_1,\varphi_2,\cdots,\varphi_{n-1}\in Der^1(A)$. In other words, $A_t$ is a trivial deformation if and only if the following equations hold
\begin{align}\label{Exp:2}
\varphi_1(x_1,\ldots,x_n)=\sum^n_{i=1}[x_1,\ldots,N(x_i),\ldots,x_n]-N([x_1,\ldots,x_n]),
\end{align}
for any $2\leq k\leq n-1$,
\begin{align}\label{Exp:4}
\nonumber
&\big(\varphi_k+N\varphi_{k-1}\big)(x_1,\ldots,x_n)\\
&=\sum_{1\leq i_1<i_2<\ldots<i_{k}\leq n}[x_1,\ldots,N(x_{i_1}),\ldots,N(x_{i_2}),\ldots,N(x_{i_{k}}),\ldots,x_n],
\end{align}
\begin{align}\label{Exp:3}
N(\varphi_{n-1}(x_1,\ldots,x_n))=[N(x_1),N(x_2)\ldots,N(x_n)],
\end{align}
for any $1\leq k\leq n-1$,
\begin{align}\label{Exp:6}
\nonumber
&\sigma_{\varphi_{k}}(x_1,\ldots,x_{n-1})\\
&= \sum_{1\leq i_1<i_2<\ldots<i_{k}\leq n-1}a(x_1,\ldots,N(x_{i_1}),\ldots,N(x_{i_2}),\ldots,N(x_{i_{k}}),\ldots,x_{n-1}).
\end{align}

\noindent We observe that from equations \eqref{Exp:2} and \eqref{Exp:4} the expression \eqref{Exp:6} hold. 

Let us define $n$-ary brackets $[~,\ldots,~]^k_N:\wedge^n A\rightarrow A$ by induction as follows: \\
for $k=1$,
\begin{equation}\label{1stbraket}
[x_1,x_2,\ldots,x_n]^1_N=\sum_{i=1}^n[x_1,\ldots,N(x_i),\ldots,x_n]-N[x_1,x_2,\ldots,x_n],
\end{equation}
for $2\leq k\leq n-1$, let us define
\begin{align}\label{kthbracket}
\nonumber
&[x_1,x_2,\ldots,x_n]^k_N\\
=&\sum_{1\leq i_1< i_2<\ldots< i_k\leq n}[x_1,\ldots,N(x_{i_1}),\ldots,N(x_{i_k}),\ldots,x_n]-N[x_1,x_2,\ldots,x_n]^{k-1}_N.
\end{align}
We call the vector bundle map $N:A\rightarrow A$ to be a ``Nijenhuis operator" on the Filippov algebroid $(A,[~,\ldots,~],a)$ if the following expression holds.
\begin{equation}\label{Nijenhuis condition}
[Nx_1,Nx_2,\ldots,Nx_n]=N\big([x_1,x_2,\ldots,x_n]^{n-1}_N\big)
\end{equation}
\begin{remark}
For $n=2$, the above definition coincides with the Nijenhuis operators on a Lie algebroid (see \cite{LeftSym}). Moreover, in the particular case when $A$ is a vector bundle over a point, the Filippov algebroid is simply an $n$-Lie algebra and the Nijenhuis operator coincides with the definition as in \cite{LSBai}.
\end{remark}

\begin{remark}
Let $A$ be a Filippov algebroid and $\rho$ be a representation of $A$ on a vector bundle $E$, then a bundle map $T:E\rightarrow A$ is an $\mathcal{O}$-operator if and only if 
$\tilde{T}$ is a Nijenhuis operator on the Filippov algebroid $A\rtimes_{\rho} E$, where 
$$\tilde{T}=\begin{bmatrix}
   0 & T \\
    0  & 0
\end{bmatrix}.$$  
\end{remark}

The following theorem characterizes the trivial deformations (of order $(n-1)$) of a Filippov algebroid (of order $n$) in terms of Nijenhuis operators. 

\begin{theorem}
A trivial deformation of a Filippov algebroid $(A,[~,\ldots,~],a)$ induces a Nijenhuis operator. Conversely, if $N:A\rightarrow A$ is a Nijenhuis operator on the Filippov algebroid $(A,[~,\ldots,~],a)$, then the cochains $\varphi_1,\ldots,\varphi_{n-1}\in Der^1(A)$ defined by
 $$\varphi_i(x_1,\ldots,x_n)=[x_1,\ldots,x_n]^i_N\quad\mbox{for } 1\leq i\leq n-1$$ 
generate a trivial deformation of the Filippov algebroid $(A,[~,\ldots,~],a)$.
\end{theorem}

\begin{proof}
First, let us assume that $A_t$ is a trivial deformation of $(A,[~,\ldots,~],a)$. Then using equations \eqref{Exp:2}-\eqref{Exp:3} and using the notations in the previous discussion, it is clear that we obtain a vector bundle map $N:A\rightarrow A$ such that 
$$[Nx_1,Nx_2,\ldots,Nx_n]=N\big([x_1,x_2,\ldots,x_n]^{n-1}_N\big).$$

Conversely, let $N:A\rightarrow A$ be a Nijenhuis operator on $(A,[~,\ldots,~],a)$, i.e. equation \eqref{Nijenhuis condition} holds. We consider the cochains $\varphi_1,\ldots,\varphi_{n-1}\in Der^1(A)$ defined by
 $$\varphi_i(x_1,\ldots,x_n)=[x_1,\ldots,x_n]^i_N\quad\mbox{for } 1\leq i\leq n-1$$ 
By a straightforward calculation one can verify that 
$\varphi_i$ satisfies the following identities (for more details see the $n$-Lie algebra case discussed in Theorem $1$, \cite{LSBai}).
\begin{enumerate}
\item $\delta_F(\varphi_1)=0$;
\item $\delta_F(\varphi_k)+\frac{1}{2}\sum\limits_{i=1}^{k-1}[\varphi_i,\varphi_{k-i}]=0$ for all $1< k< n$;
\item $\sum\limits_{i=k-n+1}^{n-1}[\varphi_i,\varphi_{k-i}]=0$ for any $n\leq k\leq 2n-2$.
\end{enumerate} 
Therefore, by Theorem \ref{Infinitesimal} the cochains $\varphi_1,\ldots,\varphi_{n-1}\in Der^1(A)$ generate a deformation $A_t$ of order $(n-1)$ of the Filippov algebroid. 

Next, we show that this deformation is trivial. We consider the map $\Phi_t: A_t\rightarrow A$ given by $\Phi_t=Id+tN$. Then
\begin{align*}
&\Phi_t([x_1,\ldots,x_n]_t)\\
=&\Phi_t\big([x_1,\ldots,x_n]+\sum\limits_{i=1}^{n-1}t^i\varphi_i(x_1,\ldots,x_n)\big)\\
=&\Phi_t\Big([x_1,\ldots,x_n]+\sum\limits_{i=1}^{n-1} t^i[x_1,\ldots,x_n]^i_N)\Big)\\
=&[x_1,\ldots,x_n]+tN[x_1,\ldots,x_n]+\sum\limits_{i=1}^{n-1} t^i[x_1,\ldots,x_n]^i_N+\sum\limits_{i=1}^{n-1} t^{i+1}N([x_1,\ldots,x_n]^i_N)\\
\end{align*} 
and 
\begin{align*}
&[\Phi_t(x_1),\ldots,\Phi_t(x_n)]\\
=&[x_1,\ldots,x_n]+\sum_{k=1}^nt^k\sum_{1\leq i_1<i_2<\ldots<i_{k}\leq n}[x_1,\ldots,N(x_{i_1}),\ldots,N(x_{i_2}),\ldots,N(x_{i_{k}}),\ldots,x_n]
\\
\end{align*}
Thus, using the definition of the $n$-ary brackets $[~,\ldots,~]^i_N$ and the Nijenhuis condition \eqref{Nijenhuis condition}, we get the following. 
$$\Phi_t([x_1,\ldots,x_n]_t)=[\Phi_t(x_1),\Phi_t(x_2),\ldots,\Phi_t(x_n)].$$ 

Next, we show by induction that the symbol $\sigma_{\varphi_k}$ of $1$-cochain $\varphi_k=[~,\ldots,~]^k_N$ for all $1\leq k\leq n-1$ is given by 
\begin{align}\label{ksymbols}
\nonumber
&\sigma_{\varphi_k}(x_1,x_2,\ldots,x_{n-1})\\
=&\sum_{1\leq i_1< i_2<\ldots< i_k\leq n-1}a\big(x_1,\ldots,N(x_{i_1}),\ldots,N(x_{i_2}),\ldots\ldots,N(x_{i_k}),\ldots,x_{n-1}\big)
\end{align}
For $k=1$, from the equation \eqref{1stbraket}, we get 
\begin{align}\label{1stbraketII}
\nonumber
[x_1,x_2,\ldots,fx_n]^1_N=&\sum_{i=1}^{n-1}[x_1,\ldots,N(x_i),\ldots,fx_n]+[x_1,\ldots,x_i,\ldots,N(fx_n)]\\
&-N[x_1,x_2,\ldots,fx_n],
\end{align}
By equation \eqref{1stbraketII}, we obtain that the symbol $\sigma_{\varphi_1}$ of $\varphi_1$ is given by 
\begin{equation}\label{Sym1}
\sigma_{\varphi_1}(x_1,x_2,\ldots,x_{n-1})=\sum_{i=1}^{n-1}a(x_1,\ldots,N(x_i),\ldots,x_{n-1})
\end{equation}
i.e., equation \eqref{ksymbols} holds for $k=1$. Now, let us assume that for $k=j-1$, the equation \eqref{ksymbols} holds. Thus,
\begin{align}\label{j-1symbols}
\nonumber
&\sigma_{\varphi_{j-1}}(x_1,x_2,\ldots,x_{n-1})\\
=&\sum_{1\leq i_1< i_2<\ldots< i_{j-1}\leq n-1}a\big(x_1,\ldots,N(x_{i_1}),\ldots,N(x_{i_2}),\ldots\ldots,N(x_{i_{j-1}}),\ldots,x_{n-1}\big).
\end{align}
Then, we show that the equation \eqref{ksymbols} also holds for $k=j$.
We use the expression \eqref{kthbracket} for $k=j$ to obtain the following identity.
\begin{align}\label{kthbracketII}
\nonumber
&[x_1,x_2,\ldots,fx_n]^j_N\\\nonumber
=&\sum_{1\leq i_1< i_2<\ldots< i_j<n}[x_1,\ldots,N(x_{i_1}),\ldots,N(x_{i_2}),\ldots\ldots,N(x_{i_j}),\ldots,fx_n]\\ \nonumber
&+\sum_{1\leq i_1< i_2<\ldots< i_{j-1}<n}[x_1,\ldots,N(x_{i_1}),\ldots,N(x_{i_1}),\ldots\ldots,N(x_{i_{j-1}}),\ldots,N(fx_n)]\\
&-N[x_1,x_2,\ldots,x_n]^{j-1}_N.
\end{align}
From equations \eqref{j-1symbols}, \eqref{kthbracketII}, it follows that the equation \eqref{ksymbols} holds for $k=j$. Therefore, by induction equation \eqref{ksymbols} holds for any $1\leq k\leq n-1$. Next, by using identity \eqref{ksymbols} it immediately follows that 
$$a(\Phi_t(x_1),\Phi_t(x_2),\ldots,\Phi_t(x_{n-1}))=a_t(x_1,x_2,\ldots,x_{n-1}).$$
Hence, the deformation generated by $1$-cochains $\varphi_1,\varphi_2,\ldots,\varphi_{n-1}$ is a trivial deformation.
\end{proof}


\section{Finite order deformations}
In this section, we consider finite order deformations of Filippov algebroids in order to interpret the equivalence classes of infinitesimal deformations and rigid objects in terms of $H^2_F$. In the sequel, we discuss the problem of extending a finite order deformation to a deformation of the next higher order.  
\begin{definition}
Let  $(A,[~,\ldots,~],a)$ be a Filippov algebroid and it corresponds to $1$-multiderivation $\varphi\in Der^1(A)$ satisfying $[\varphi,\varphi]=0$. Then one-parameter deformation of order $k$ of $\varphi$ is given by a $t$-parametrized family 
$$\textstyle{\varphi_t=\varphi_0+\sum^k_{i=1}t^i\varphi_i }, \quad\mbox{ for }  \varphi_0=\varphi,~\varphi_i\in Der^1(A),\mbox{~and}~~ t\in\mathbb{R},$$ 
such that $$[\varphi_t,\varphi_t]=0 \quad\mbox{modulo}~~ t^{k+1}.$$
i.e.,
$$\sum\limits_{i+j=r; ~i,j\geq 0}[\varphi_i,\varphi_j]=0\quad\mbox{for}~~r=0,1,\ldots,k.$$
\end{definition}

Let $\textstyle{\varphi_t=\varphi_0+\sum^k_{i=1}t^i\varphi_i }$ be a deformation of order $k$. Then the 2-cochain $\varphi_1$ is called the infinitesimal of the deformation $\varphi_t$. More generally, if $\varphi_i=0$ for $1\leq i\leq (m-1)$ and $\varphi_m$ is a non-zero cochain for $m\leq k$, then $\varphi_m$ is called the $m$-infinitesimal of the deformation $\varphi_t$. On comparing the coefficients of 
$t$ in the expression
$$[\varphi_t,\varphi_t]=0 \quad\mbox{modulo}~~ t^{k+1},$$ 
we get the following proposition.

\begin{proposition}\label{infinitesimal}
The infinitesimal of the deformation $\varphi_t$ is a $2$-cocycle in $C^2_F(A)$. More generally, the $m$-infinitesimal is a $2$-cocycle. 
\end{proposition}
\begin{definition}
Two deformations $\varphi_t$ and $\tilde{\varphi}_t$ of order $k$ are said to be equivalent if we have an automorphism $$\Phi_t : A \rightarrow A~~\mbox{ defined as}~~
\Phi_t=Id+\sum_{i=1}^k t^i\phi_i,$$
where $\phi_i:A \rightarrow A$ is a linear map, for $1\leq i \leq k$ such that 
\begin{equation}\label{equivalence condition}
\tilde{\varphi}_t(x_1,x_2,\ldots,x_n)=\Phi_t^{-1}\varphi_t(\Phi_t x_1, \Phi_t x_2,\ldots,\Phi_t x_n)\quad\mbox{modulo}~~ t^{k+1}.
\end{equation}
Here, the map $\Phi_t$ is invertible modulo $t^{k+1}$ and we denote the inverse by $\Phi_t^{-1}$. 
\end{definition}

On comparing the coefficients of $t$ from both sides of the equation \eqref{equivalence condition}, we get 
$$
\varphi_1-\tilde{\varphi}_1=[\varphi,\phi_1].
$$
Consequently, we obtain the following result.
\begin{theorem}\label{equivalent deformations}
The infinitesimals of equivalent finite order deformations of a Filippov algebroid $(A,[~,\ldots,~],a)$ belong to the same cohomology class in $H^2_F(A)$.
\end{theorem}

A deformation of order $1$ of a Filippov algebroid is called an infinitesimal deformation. The next result interprets the equivalence classes of infinitesimal deformations of Filippov algebroids in terms of $H^2_F$. 
\begin{theorem}
Let $(A,[~,\ldots,~],a)$ be a Filippov algebroid. There is a bijective correspondence between the second cohomology space $H^2_F(A)$ and the equivalence classes of infinitesimal deformations of $(A,[~,\ldots,~],a)$.
\end{theorem}
\begin{proof}

Let $\varphi_1$ be a $2$-cocycle in $C^2_F(A)$. Then let us consider a map 
$$\varphi_t:=\varphi+t\varphi_1:A\rightarrow A.$$
Since $\delta_F(\varphi_1)=0$, it follows that $\varphi_t$ is an infinitesimal deformation of $\varphi$. Let $\tilde{\varphi}_1$ be a $2$-cocyle, cohomologous to $\varphi_1$ in $H^2_F(A)$. i.e., there exists $\phi_1\in C^1_F(A)$ such that ${\varphi}_1 - \tilde{\varphi}_1=\delta_F(\phi_1)$. Then, the infinitesimal deformation corresponding to the $2$-cocycle $\tilde{\varphi}_1$ is equivalent to the infinitesimal deformation $\varphi_t$. The equivalence between these two deformations is given by the map $\Phi_t:=Id+t\phi_1$. Therefore, we have a correspondence between $H^2_F(A)$ and the space of equivalence classes of infinitesimal deformations of $\varphi$, given by a map sending $[\varphi_1]\mapsto[\varphi_t]$.

Next, we show that the above mentioned correspondence is bijective. From Proposition \ref{infinitesimal}, it is clear that the correspondence is surjective. More precisely, for any equivalence class of infinitesimal deformations $[\varphi_t]$, we get a cohomology class $[\varphi_1]\in H^2_F(A)$. To show that the correspondence is injective, let us assume that the cohomology classes $[\varphi_1], [\tilde{\varphi}_1]\in H^2_F(A)$ correspond to the same equivalence class of infinitesimal deformations of $\varphi$. Let $\varphi_t:=\varphi+t{\varphi}_1$ and $ \tilde{\varphi}_t:=\varphi+t\tilde{\varphi}_1$ be the infinitesimal deformations corresponding to the $2$-cocycles $\varphi_1$ and $\tilde{\varphi}_1$, respectively. By assumption, infinitesimal deformations $\tilde{\varphi}_t$ and ${\varphi}_t$ are equivalent. Therefore, by Theorem \ref{equivalent deformations}, it follows that $ [\varphi_1]=[\tilde{\varphi}_1]$. Thus, the correspondence is bijective.
\end{proof}

Let us observe that the deformation $\varphi_0=\varphi$ can be considered as a deformation of order $k$, for $k\geq 1$. With this observation, we define rigid Filippov algebroids as follows:

\begin{definition}
A deformation of order $k$ of a Filippov algebroid is a trivial deformation if it is equivalent to the deformation $\varphi_{0}=\varphi$. We call a Filippov algebroid to be `rigid' if every finite order deformation is trivial.
\end{definition}

\begin{theorem}
Let $(A,[~,\ldots,~],a)$ be a Filippov algebroid. If $H^2_F(A)=0$, then the Filippov algebroid $(A,[~,\ldots,~],a)$ is rigid.
\end{theorem}
\begin{proof}
Let $\varphi_t$ be a deformation of order $k$ with $m$-infinitesimal $\varphi_m$, for $1\leq m\leq k$. 
From Proposition \ref{infinitesimal}, the $m$-infinitesimal $\varphi_m$ is a $2$-cocycle. Since $H^2_F(A)=0$, there exists 1-cochain $\Psi \in C^1_{F}(A)$ such that  $\delta_F(\Psi)=\varphi_m$. Let us define
$$
\Phi_t=Id+t^m \Psi  \quad\mbox{and}\quad\tilde{\varphi}_t(x_1,x_2,\ldots,x_n)=\Phi_t^{-1}\varphi_t(\Phi_t x_1, \Phi_t x_2,\ldots,\Phi_t x_n)~~ \mbox{modulo}~ t^{k+1}.
$$
It is clear that $\tilde{\varphi}_t$ is a deformation of order $k$ and it is equivalent to the deformation ${\varphi}_t$. Next, by comparing the coefficients of $t^m$, we get that
$$
\tilde{\varphi}_m-\varphi_m=-[\varphi,\Psi]=-\delta_F(\Psi),
$$
which implies that $\tilde{\varphi}_m=0$. Thus, ${\varphi}_t$ is equivalent to a deformation $\tilde{\varphi}_t:=\varphi+\sum\limits_{i=m+1}^k t^i\tilde{\varphi}_i$. We can repeat the argument to show that ${\varphi}_t$ is equivalent to $\varphi_0=\varphi$.
\end{proof}
\subsection{Obstructions}

Let $(A,[~,\ldots,~],a)$ be a Filippov algebroid and it corresponds to $1$-multiderivation $\varphi\in Der^1(A)$ satisfying $[\varphi,\varphi]=0$. We associate a $3$-cocycle in $H^3_F(A)$ (called obstruction cocycle) to any deformation of order $k$ of $\varphi$. We show that the deformation of order $k$ extends to a deformation of order $k+1$ if and only if the corresponding obstruction class vanishes.

\begin{definition}
Let $\textstyle{\varphi_t=\varphi_0+\sum^k_{i=1}t^i\varphi_i }$ be a deformation of order $k$ of Filippov algebroid $(A,[~,\ldots,~],a)$. Accordingly, we say that $\varphi_t$ extends to a deformation of order $k+1$ if there exists a 2-cochain $\varphi_{k+1}\in C^2_{F}(A)$ such that $$\tilde{\varphi}_t=\varphi_t+t^{k+1} \varphi_{k+1} $$ is a deformation of order $k+1$.
\end{definition}

If $\varphi_t$ is a deformation of order $k$, then on comparing the coefficients of $t^r$ in the expression $[\varphi_t,\varphi_t]=0$ modulo $t^{k+1}$, we get the following identities
$$\sum_{i+j=r;~~i,j\geq 0}[\varphi_i,\varphi_j]=0\quad \mbox{for ~}r=0,1,\ldots,k.$$

Thus, for an element $\varphi_{k+1}\in C^2_{F}(A)$, the map $\tilde{\varphi}_t:=\varphi_t+t^{k+1}\varphi_{k+1}$ is an extension of $\varphi_t$ if
$$\sum_{i+j=k+1;~~i,j\geq 0}[\varphi_i,\varphi_j]=0.$$

\begin{definition}
Let $\varphi_t$ be a deformation of $\varphi$ of order $k$. Let us consider the $3$-cochain $\Theta_F \in C^3_{F}(A)$ defined as follows
\begin{equation}\label{Obst}
\Theta_F =-1/2\sum_{i+j=k+1;~i,j>0}[\varphi_i,\varphi_j].
\end{equation}
The 3-cochain $\Theta_F$ is called the obstruction cochain for extending the deformation of $\varphi$ of order $k$ to a deformation of order $k+1$. 
By equation \eqref{Obst}, and using graded Jacobi identity of the bracket $[~,~]$ it follows that $\Theta_F$ is a 3-cocycle.
\end{definition}

\begin{theorem}\label{hom-Obst}
Let $\varphi_t$ be a deformation of $\varphi$ of order $k$. Then $\varphi_t$ extends to a deformation of order $k+1$ if and only if the cohomology class of $\Theta_F$ vanishes.

\begin{proof}
Suppose that a deformation $\varphi_t$ of order $k$ extends to a deformation of order $k+1$. Thus, there exists an element $\varphi_{k+1}\in C^2_{F}(A)$ such that 
$\tilde{\varphi_t}=\varphi_t+t^{k+1} \varphi_{k+1}$ is an extension of $\varphi_t$, i.e., 
$$\sum_{i+j=k+1;~~i,j\geq 0}[\varphi_i,\varphi_j]=0.$$
Consequently, we get $\Theta_F=\delta_F(\varphi_{k+1})$. Hence, the cohomology class of $\Theta_F$ vanishes.

Conversely, let us assume that $\Theta_F$ is a coboundary. Suppose that
$$
\Theta_F=\delta(\varphi_{k+1}) 
$$
for some $2$-cochain $\varphi_{k+1}$. Define a map $\tilde{\varphi_t}:\D\times \Gamma(A)\rightarrow \Gamma(A)$ as follows
$$
\tilde{\varphi_t}=\varphi_t+t^{k+1}\varphi_{k+1}.
$$
Then,
$$\Theta_F=-1/2\sum_{i+j=k+1;~i,j>0}[\varphi_i,\varphi_j]=\delta(\varphi_{k+1}).$$
Since, $\delta_F(\varphi_{k+1})=[\varphi,\varphi_{k+1}]$, we obtain the following expression
$$\sum_{i+j=k+1;~i,j\geq 0}[\varphi_i,\varphi_j]=0.$$ 
This, in turn, implies that $\tilde{\varphi_t}$ satisfies the identity: $[\tilde{\varphi_t},\tilde{\varphi_t}]=0$ modulo $t^{k+2}$. Therefore, the deformation $\varphi_t$ extends to a deformation $\tilde{\varphi_t}$ of order $k+1$.
\end{proof}

\end{theorem}

\begin{corollary}
Let $(A,[~,\ldots,~],a)$ be a Filippov algebroid. If $H^3_{F}(A)=0$, then every finite order deformation of the Filippov algebroid extends to higher order deformations.
\end{corollary}

\noindent {\em Acknowledgements.} The research of S. K. Mishra is supported by the NBHM postdoctoral fellowship. The author thanks NBHM for its support.

\end{document}